\documentclass{amsart}

\usepackage{graphicx}
\usepackage[latin1]{inputenc}

\newtheorem{thm}{Theorem}[section]

\newtheorem{lem}[thm]{Lemma}

\newtheorem{exa}[thm]{Example}
\newtheorem{dfn}[thm]{Definition}
\newtheorem{rem}[thm]{Remark}
\numberwithin{equation}{section}

\begin{document}

\title[]{On the generalization of Conway algebra}

\author{Seongjeong Kim}

\address{Department of Fundamental Sciences, Bauman Moscow State Technical University, Moscow, Russia \\
ksj19891120@gmail.com}


\begin{abstract}
In \cite{PrzytyskiTraczyk} J.H.Przytyski and P.Traczyk introduced an algebraic structure, called {\it a Conway algebra,} and constructed an invariant of oriented links, which is a generalization of the Homflypt polynomial invariant. On the other hand, in \cite{KauffmanLambropoulou} L. H. Kauffman and S. Lambropoulou introduced new 4-variable invariants of oriented links, which are obtained by two computational steps: in the first step we apply a skein relation on every mixed crossing to produce unions of unlinked knots. In the second step, we apply the skein relation on crossings of the unions of unlinked knots, which introduces a new variable. In this paper, we will introduce a generalization of the Conway algebra $\widehat{A}$ with two binary operations and we construct an invariant valued in $\widehat{A}$ by applying those two binary operations to mixed crossings and self crossings respectively. Moreover, the generalized Conway algebra gives us an invariant of oriented links, which satisfies non-linear skein relations.
\end{abstract}

\maketitle

\section{Introduction}
The construction of the Jones polynomial by V. F. R. Jones in 1984 is considered as one of accomplishments in knot theory. At first it was constructed by the Artin braid groups and their Markov equivalence by means of a Markov trace on the Temperley-Lieb algebras. The diagrammatic method proposed by L. H. Kauffman, which is known as {\it Kauffman's bracket,} makes it easy to calculate the Jones polynomial. In 1987, J. Hoste, A. Ocneanu, K. Millett, P. J. Freyd, W. B. R. Lickorish, D. N. Yetter, J. H. Przytycki and P. Traczyk discovered the 2-variable polynomial invariant {\it the Homflypt polynomial}, which is a generalization of the Jones polynomial and the Alexander polynomial. The Homflypt polynomial can be constructed by using the Ocneanu trace defined on the Iwahori-Hecke algebra of type A and can be also defined diagrammatically by a skein relation. Especially J. H. Przytycki and P. Traczyk in~\cite{PrzytyskiTraczyk} defined the Homflypt polynomial by new algebraic structure, called {\it the Conway algebra.}
In 2016 M. Chlouveraki, J. Juyumaya, K. Karvounis, S. Lambropoulou and W. B. R. Lickorish ~\cite{ChlouverakiJuyumayaKarvounisLambropoulouLickorish} studied a family of new 2-variable polynomial invariants for oriented links defined by the Markov trace on the Yokonuma-Hecke algebra and showed that these invariants are topologically equivalent to the Homflypt polynomial for knots, but not for links. The authors generalized the family of 2-variable polynomials to a new 3-variable link invariants defined by a skein relation for oriented links, which is stronger than the Homflypt polynomial. In 2017 L. H. Kauffman and S. Lambropoulou~\cite{KauffmanLambropoulou} constructed new 4-variable polynomial invariants, which are generalized from the Homflypt polynomial, the Dubrovnik polynomial and the Kauffman polynomial. Roughly speaking,  they are obtained by two computational steps: in the first step we apply a skein relation on every mixed crossings to produce unions of unlinked knots. In the second step, we apply the skein relation on self crossings of the unions of unlinked knots, which introduces a new variable.

In this paper we introduce the Conway algebra with two binary operations and construct a generalized Conway type invariant valued in the generalized Conway algebra by applying one of those two operations to mixed crossings and self crossings respectively. We will show that the generalized Conway algebra gives us an invariant of oriented links, which satisfies `non-linear' skein relation.

The paper is organized as follows: in section 2 we introduce the Conway algebra and the Conway type invariant, which are introduced in~\cite{PrzytyskiTraczyk}. In section 3 we study a generalization of the Conway algebra and the Conway type invariant, which are called {\it a generalized Conway algebra of type 1 and a generalized Conway type invariant of type 1}, respectively. And we will introduce an example of generalized Conway type invariants of type 1 with non-linear skein relations.

\section{Conway algebra and Conway type invariant}
In this section, we recall the Conway algebra and the invariant of classical links valued in the Conway algebra~\cite{PrzytyskiTraczyk}.
\begin{dfn}\cite{PrzytyskiTraczyk}\label{def-Conwayal}
Let $\mathcal{A}$ be an algebra with two binary operations $\circ$ and $/$ on $\mathcal{A}$. Let $\{a_{n}\}_{n=1}^{\infty} \subset \mathcal{A}$. The quadruple $(\mathcal{A}, \circ, /,\{a_{n}\}_{n=1}^{\infty})$ is called {\it a Conway algebra} if it satisfies the following relations:
\begin{enumerate}
\item $(a \circ b) / b = a =(a/ b)  \circ b$ for $a,b \in \mathcal{A}$,
\item $ a_{n} = a_{n} \circ a_{n+1}$ for $n \in \mathbb{N}$,
\item $(a \circ b) \circ (c \circ d) = (a \circ c) \circ (b \circ d)$ for $a,b,c,d \in \mathcal{A}$.

\end{enumerate}

\end{dfn}

\begin{rem}
The original definition of the Conway algebra~\cite{PrzytyskiTraczyk} has three more relations $(a\circ b) / (c \circ d) =  (a / c) \circ (b / d)$, $(a/ b) / (c / d) =  (a/ c) / (b / d)$ and $a_{n} = a_{n} / a_{n+1}$. By some calculations, we can show that they are obtained from the relations in definition~\ref{def-Conwayal}.
\end{rem}
Let $\mathcal{L}$ be the set of equivalence classes of oriented link diagrams modulo Reidemeister moves. In \cite{PrzytyskiTraczyk}, J. H. Przytyski and P. Traczyk defined an invariant $W$, which we call {\it the Conway type invariant valued in the Conway algebra $(\mathcal{A}, \circ, /,\{a_{n}\}_{n=1}^{\infty})$} satisfying the following properties:
\begin{enumerate}
\item For the trivial link $T_{n}$ of $n$ components, 
$$W(T_{n}) = a_{n}.$$
\item For each crossing the following skein relation holds: $$W(L_{+}) = W(L_{-}) \circ W(L_{0}),$$ where $L_{+},L_{-},L_{0}$ are the oriented Conway triple, described in Fig.~\ref{1Conwaytriple}.
\end{enumerate}

\begin{figure}[h!]
\begin{center}
 \includegraphics[width = 8cm]{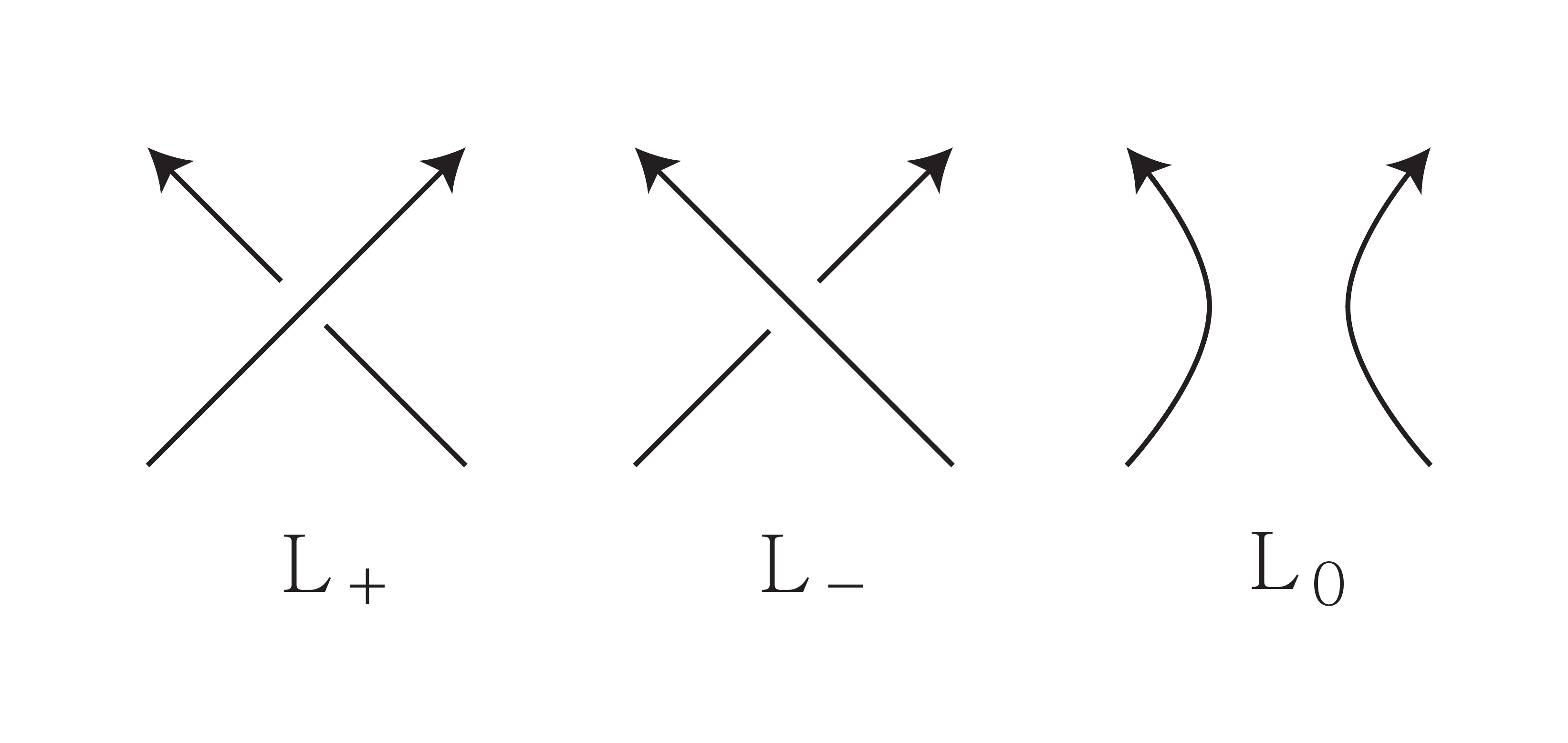}

\end{center}

 \caption{The oriented Conway triple $L_{+},L_{-}$ and $L_{0}$}\label{1Conwaytriple}
\end{figure}

 We recall the construction of the Conway type invariant, which is introduced in \cite{PrzytyskiTraczyk}: let $L = L_{1} \cup \cdots \cup L_{r}$ be an ordered oriented link diagram of $r$ components. Fix a base point $b_{i}$ on each component $L_{i}$. Suppose that we walk along the diagram $L_{1}$ according to the orientation from the base point $b_{1}$ to itself, then we walk along the diagram $L_{2}$ from the base point $b_{2}$ to itself and so on. If we pass a crossing $c$ first along the undercrossing(or overcrossing), we call $c$ {\it a bad crossing}(or {\it a good crossing}). We do crossing change for all bad crossings or splice bad crossings. To specify the crossing $c$ of $L$, we denote the diagram, in which the crossing $c$ has $sgn(c) = +1$($sgn(c) = -1$)  by $L^{c}_{+}$($L^{c}_{-}$). If a diagram is obtained from $L_{+}^{c}$ or $L_{-}^{c}$ by splicing the crossing $c$, we denote the diagram by $L_{0}^{c}$. Suppose that we meet the first bad crossing $c$ with $sgn(c) = +1$. Then by applying the skein relation we obtain
\begin{equation*}\label{selfConwayrel}
W(L^{c}_{+}) = W(L^{c}_{-}) \circ W(L^{c}_{0}).
\end{equation*}
Notice that if $c$ is a bad crossing of $L^{c}_{+}$ with respect to $b$, then the number of bad crossings of $L_{-}^{c}$ is less than the number of bad crossings of $L_{+}^{c}$ and the number of crossings of $L_{0}^{c}$ is less than the number of crossings of $L_{+}^{c}$. For the first bad crossing $c$ with $sgn(c) = -1$ with respect to $b$ we obtain 
\begin{equation*}
W(L^{c}_{-}) = W(L^{c}_{+}) / W(L^{c}_{0}),
\end{equation*}
and we repeat this process inductively on the diagrams, which are obtained from $L$ by a crossing change and a splicing, until we switch all bad crossings.
Note that if $L$ has no bad crossings, then the diagram $L$ is equivalent to the trivial link diagram, since $L$ is a descending diagram. For the trivial link $T_{n}$, $W(T_{n}) = a_{n}$.

\begin{thm}\cite{PrzytyskiTraczyk}
The Conway type invariant $W : \mathcal{L} \rightarrow \mathcal{A}$ is well defined. That is, the value of $W$ does not depend on the choice of base points and the order of links, and it is invariant under the Reidemeister moves.
\end{thm}

\begin{exa}\label{exa_imp_Conway}\cite{PrzytyskiTraczyk}
Let $\mathcal{A} = \mathbb{Z} [p^{\pm 1 }, q^{\pm 1},z^{\pm 1}]$. Define binary operations $\circ$ and $/$ by
$$a \circ b = pa + qb + z~\text{and}~ a/b = p^{-1}a - p^{-1}qb - p^{-1}z.$$
Let $\{a_{n}\}_{n=1}^{\infty}$ be a sequence defined by the recurrence formula 
$$ a_{1} = 1~\text{and}~(1-p)a_{n} = qa_{n+1}+z,$$
or by the formula
$$ a_{n} =(1-\frac{z}{(1-p-q)})(\frac{1-p}{q})^{n-1} + \frac{z}{1-p-q}.$$
Then $(\mathbb{Z} [p^{\pm 1 }, q^{\pm 1},z^{\pm 1}], \circ, /, \{a_{n}\}_{n=1}^{\infty})$ is a Conway algebra.
\end{exa}

\begin{exa}\label{exa_imp_Conway-1}\cite{PrzytyskiTraczyk}
Let $\mathcal{A} = \mathbb{Z} [v^{\pm 1 }, z^{\pm 1}]$. Define the binary operations $\circ$ and $/$ by
$$a \circ b = v^{2}a + vzb, ~ a / b = v^{-2} a - v^{-1}zb.$$
Denote $a_{n} = (\frac{v^{-1} - v}{z})^{n-1}$ for each $n$. This is obtained from the Conway algebra in Example~\ref{exa_imp_Conway} by substituting $p=v^{2},q = vz,z=0$. Moreover, the Conway type invariant $W(L)$ valued in $( \mathbb{Z} [v^{\pm 1 }, z^{\pm 1}], \circ, /, \{a_{n}\}_{n=1}^{\infty})$ is the Homflypt polynomial.
\end{exa}

\section{Generalization of Conway algebra and Conway type invariant}
In this section we introduce a generalization of the Conway algebra. 
\begin{dfn}\label{def_genConaltype1}
Let $\widehat{\mathcal{A}}$ be a set with four binary operations  $\circ,*,/$ and $//$ on $\widehat{\mathcal{A}}$. Let $\{a_{n}\}_{n=1}^{\infty} \subset \widehat{\mathcal{A}}$. The hexuple $( \widehat{\mathcal{A}}, \circ,/,*, //,\{a_{n}\}_{n=1}^{\infty})$ is called {\it a generalized Conway algebra of type $1$} if it satisfies the following conditions:
\begin{itemize}
\item[(A)] $(a \circ b) / b = (a / b) \circ b = a = (a* b) // b  = (a // b) * b$ for $a,b \in  \widehat{\mathcal{A}}$,
\item[(B)] $ a_{n} = a_{n} \circ a_{n+1}$ for $n = 1,2,\cdots,$
\item[(C)] $(a \circ b) \circ (c \circ d) = (a \circ c) \circ (b \circ d)$ for $a,b,c,d \in  \widehat{\mathcal{A}}$,
\item[(D)] $(a * b) * (c * d) = (a * c) * (b * d)$ for $a,b,c,d \in  \widehat{\mathcal{A}}$,
\item[(E)] $(a \circ b) \circ (c * d) = (a \circ c) \circ (b * d)$ for $a,b,c,d \in  \widehat{\mathcal{A}}$,
\item[(F)] $(a * b) * (c \circ d) = (a * c) * (b \circ d)$ for $a,b,c,d \in  \widehat{\mathcal{A}}$,
\item[(G)] $(a \circ b) * (c \circ d) = (a * c) \circ (b * d)$ for $a,b,c,d \in  \widehat{\mathcal{A}}$.
\end{itemize}

\end{dfn}

\begin{rem}
Let $( \widehat{\mathcal{A}}, \circ,/,*,//,\{a_{n}\}_{n=1}^{\infty})$ be a generalized Conway algebra of type 1. The quadruple $(\widehat{\mathcal{A}}, \circ, /,\{a_{n}\}_{n=1}^{\infty})$ is a Conway algebra, and hence the Conway type invariant can be defined on $(\widehat{\mathcal{A}}, \circ, /,\{a_{n}\}_{n=1}^{\infty})$.
\end{rem}

\begin{thm}\label{Main_thm}
Let $\mathcal{L}$ be the set of equivalence classes of oriented link diagrams modulo Reidemeister moves. Let $(\widehat{\mathcal{A}}, \circ,/,*, //,\{a_{n}\}_{n=1}^{\infty})$ be a generalized Conway algebra of type $1$. Then there uniquely exists the invariant of classical oriented links $\widehat{W} : \mathcal{L} \rightarrow \widehat{\mathcal{A}}$ satisfying the following rules:
\begin{enumerate}
\item For self crossings $c$ the following relation holds: 
\begin{equation}\label{selfConwayrel}
\widehat{W}(L_{+}^{c}) = \widehat{W}(L_{-}^{c}) \circ \widehat{W}(L_{0}^{c}).
\end{equation}
\item For mixed crossings $c$ the following relation holds: 
\begin{equation}\label{mixedConwayrel}
\widehat{W}(L_{+}^{c}) = \widehat{W}(L_{-}^{c}) * \widehat{W}(L_{0}^{c}).
\end{equation}
\item Let $T_{n}$ be a trivial link of $n$ components. Then
\begin{equation}
\widehat{W}(T_{n}) = a_{n}.
\end{equation}
\end{enumerate}
We call $\widehat{W}$ {\it a generalized Conway type invariant of type $1$ valued in $(\widehat{\mathcal{A}}, \circ,/,*,//,\{a_{n}\}_{n=1}^{\infty})$.}
\end{thm}

\textbf{Construction of $\widehat{W}$.} First we will define $\widehat{W}$ for ordered oriented link diagrams. Let $L = L_{1} \cup \cdots \cup L_{r}$ be an ordered oriented link diagram of $r$ components. Fix a base point $b_{i}$ on each component $L_{i}$. Suppose that we walk along the diagram $L_{1}$ according to the orientation from the base point $b_{1}$ to itself, then we walk along the diagram $L_{2}$ from the base point $b_{2}$ to itself and so on. If we pass a crossing $c$ first along the undercrossing(or overcrossing), we call $c$ {\it a bad crossing}(or {\it a good crossing}) with respect to the base points $b = \{b_{1}, \cdots, b_{r}\}$. Now we perform the crossing change or splicing for all bad crossings. Denote the value of  $\widehat{W}$ for $L$ corresponding to the base points $b$ by $\widehat{W}_{b}(L)$. Suppose that we meet the first bad crossing $c$. If it is a self crossing, we apply the skein relation on $c$ with the following property:
\begin{equation*}
\widehat{W}_{b}(L_{+}^{c}) = \widehat{W}_{b}(L_{-}^{c}) \circ \widehat{W}_{\tilde{b}}(L_{0}^{c}),
\end{equation*}
where $\tilde{b} = \{b_{1}, \cdots, b_{j-1},b_{j},b'_{j},b_{j+1}, \cdots, b_{n},b_{n+1}\}$ and $b'_{j}$ is a chosen base point near the place of the crossing $c$ on the component, which is appeared by splicing the self crossing $c$ of $L_{j}$, see~Fig.~\ref{2splicing}.
\begin{figure}[h!]
\begin{center}
 \includegraphics[width = 8cm]{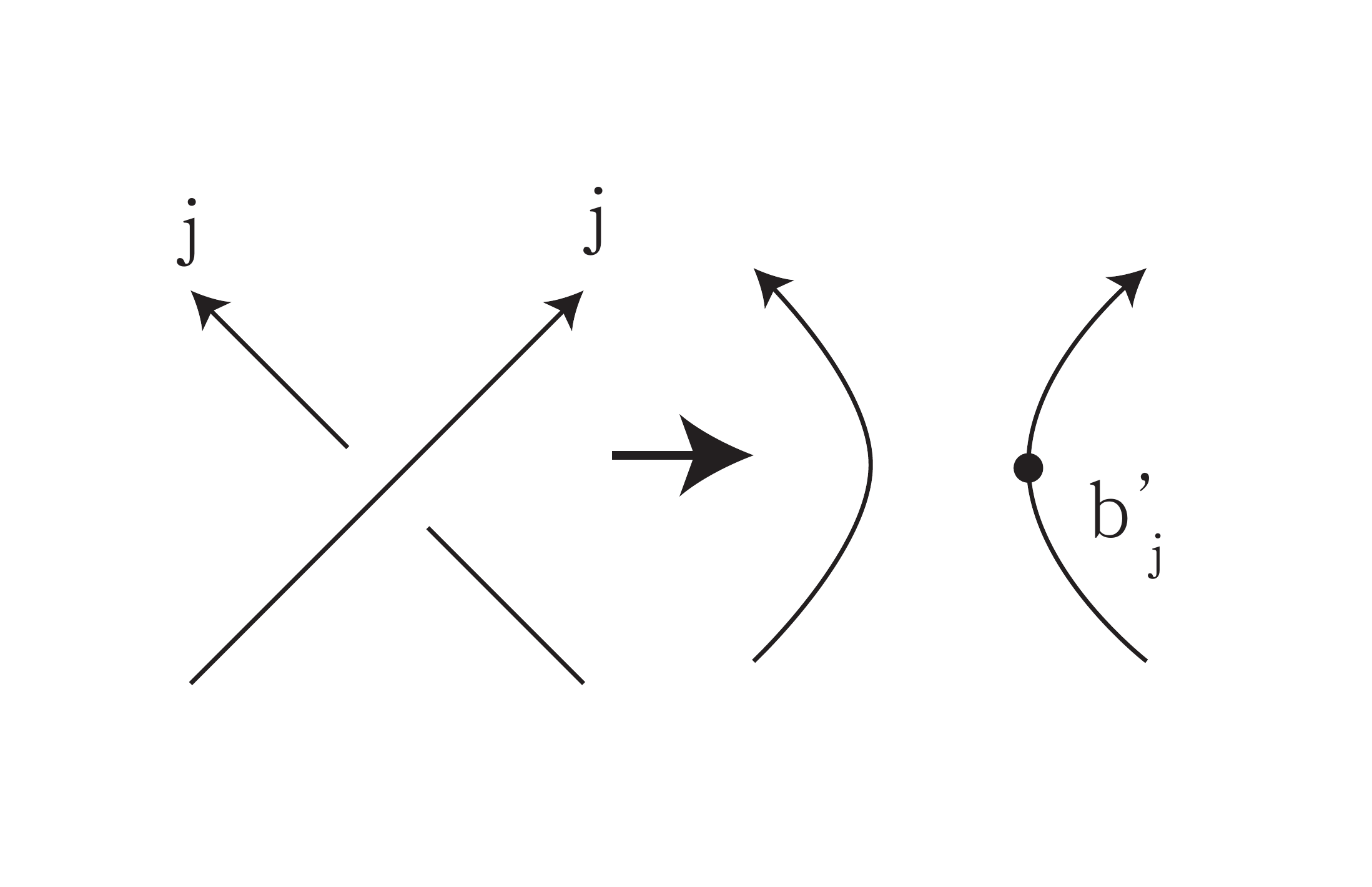}

\end{center}

 \caption{The choice of a base point $b'_{j}$}\label{2splicing}
\end{figure}

If it is a mixed crossing between two components $L_{i}$ and $L_{j}$, we apply the skein relation on $c$ with the following property:
\begin{equation*}
\widehat{W}_{b}(L_{+}^{c}) = \widehat{W}_{b}(L_{-}^{c}) *\widehat{W}_{\tilde{b}}(L_{0}^{c}),
\end{equation*}
where $\tilde{b} = \{b_{1}, \cdots, b_{j-1},b_{j+1}, \cdots, b_{n}\}$.
As an abuse of notation we write $\tilde{b} = b$, if it does not cause confusion. Notice that if $c$ is a positive crossing, then the number of bad crossings of $L_{-}^{c}$ is less than the number of bad crossings of $L_{+}^{c}$ and the number of crossings of $L_{0}^{c}$ is less than the number of crossings of $L_{+}^{c}$. We apply those relations to the first bad crossings of $L_{-}^{c}$ and $L_{0}^{c}$ inductively until we switch all bad crossings. If $c$ is a negative crossing, then we apply those relations to the first bad crossings of $L_{+}^{c}$ and $L_{0}^{c}$ inductively until we switch all bad crossings.
 If $L = L_{1} \cup \cdots \cup L_{r}$ has no bad crossings, then we define $\widehat{W}_{b}(L) = a_{n}$.

\begin{proof}
Let $\mathcal{L}_{k}$ be the set of all ordered oriented link diagrams such that diagrams in $\mathcal{L}_{k}$ have no more than $k$ crossings. We will show that $\widehat{W}_{b}(L)$ is an invariant by the following steps: for every $k = 0,1,\cdots$ and, for diagrams $L  \in \mathcal{L}_{k}$ with base points $b = \{b_{1}, \cdots, b_{r}\}$,
\begin{itemize}
\item[(a)] the mapping $\widehat{W}_{b}(L)$ is well defined,
\item[(b)] the relation (\ref{selfConwayrel}) and  (\ref{mixedConwayrel}) hold for every crossing $c$,
\item[(c)] the value of $\widehat{W}_{b}(L)$ does not depend on the choice of base points,
\item[(d)] the value of $\widehat{W}_{b}(L)$ is invariant under Reidemeister moves, which do not make the number of crossings more than $k$,
\item[(e)] the value of $\widehat{W}_{b}(L)$ does not depend on the order of components,
\end{itemize}
by the induction on $k$.
For $k=0$, it is clear that $\widehat{W}_{b}$ satisfies (a), (b), (c), (d) and (e).
Suppose that (a),(b),(c),(d) and (e) are true on $\mathcal{L}_{k}$ for $k\geq 0$. Let $L = L_{1} \cup \cdots \cup L_{r}$ be an ordered oriented link diagram in $\mathcal{L}_{k+1}$. Now let us fix base points $b = \{b_{1}, \cdots, b_{r}\}$ on $L$. \\

\textbf{(a) The mapping $\widehat{W}_{b}(L)$  is well defined.}
 If every crossing of $L$ is good with respect to $b$, it is clear. Now we use the second induction on the number of bad crossings with respect to $b$. Assume that $\widehat{W}_{b}(L)$ is well defined on diagrams $L$ in $\mathcal{L}_{k+1}$, which have $m$ bad crossings with respect to $b$. Let $L$ be a diagram with $k+1$ crossings, of which there are $m+1$ bad crossings with respect to fixed base points $b$. Let $c$ be the first bad crossing of $L$ with respect to $b$. If $c$ is a positive self crossing, by applying the relation (\ref{selfConwayrel}) to $c$, we obtain the equation 
$$\widehat{W}_{b}(L^{c}_{+}) =  \widehat{W}_{b}(L^{c}_{-}) \circ \widehat{W}_{\tilde{b}}(L^{c}_{0}).$$
Since $L^{c}_{-}$ has $m$ bad crossings with respect to $b$ and $L^{c}_{0}$ has $k$ crossings, that is, $L^{c}_{0} \in \mathcal{L}_{k}$, $\widehat{W}_{b}(L^{c}_{-})$ and $\widehat{W}_{\tilde{b}}(L^{c}_{0})$ are well-defined by the assumptions of the second and the first inductions respectively, and hence $\widehat{W}_{b}(L)$ is well defined.
If $c$ is a positive mixed crossing, by applying the relation (\ref{mixedConwayrel}) to $c$, we obtain the equation 
$$\widehat{W}_{b}(L^{c}_{+}) =  \widehat{W}_{b}(L^{c}_{-}) * \widehat{W}_{\tilde{b}}(L^{c}_{0}).$$
Analogously we can show that  $\widehat{W}_{b}(L)$ is well defined. Similarly we can show that $\widehat{W}_{b}(L)$ is well defined when the first bad crossing is negative and hence $\widehat{W}_{b}(L)$ is well defined for diagrams $L$ with $k+1$ crossings, of which there are $m+1$ bad crossings. Therefore the mapping $\widehat{W}_{b}$ is well defined on $\mathcal{L}_{k+1}$.

\textbf{(b) The relation (\ref{selfConwayrel}) and (\ref{mixedConwayrel}) hold for every crossing $c$.} To prove the statement (b) we will use the second induction on the number of bad crossings with respect to $b$. Suppose that there are no bad crossings of $L$ with respect to $b$. If $c$ is a positive self crossing of $L^{c}_{+}$, then $L^{c}_{-}$ has only one bad crossing. Since $c$ is the only one bad self crossing, $c$ is the first bad crossing with respect to $b$. By the construction of $\widehat{W}_{b}$ the equation holds
\begin{equation*}
\widehat{W}_{b}(L^{c}_{-}) = \widehat{W}_{b}(L^{c}_{+}) / \widehat{W}_{b}(L^{c}_{0}),
\end{equation*} 
and hence 
\begin{equation*}
\widehat{W}_{b}(L^{c}_{+}) = \widehat{W}_{b}(L^{c}_{-}) \circ \widehat{W}_{b}(L^{c}_{0}).
\end{equation*} 
Similarly we can show that the equation (\ref{selfConwayrel})
holds for negative self crossings $c$. Analogously we can show that the equation (\ref{mixedConwayrel})
\begin{equation*}
\widehat{W}_{b}(L^{c}_{+}) = \widehat{W}_{b}(L^{c}_{-})* \widehat{W}_{b}(L^{c}_{0})
\end{equation*}
holds for mixed crossings $c$.
Suppose that the statement (b) holds for ordered oriented link diagrams $L$ in $\mathcal{L}_{k+1}$ with $m \ge 0$ bad crossings with respect to fixed base points $b$. Let $L$ be an ordered oriented link diagram in $\mathcal{L}_{k+1}$ with $m+1$ bad crossings with respect to fixed base points $b$. For crossings $c$ and $c'$ of $L$, let us denote the diagram with states $\alpha$ and $\beta$ of crossings $c$ and $c'$ respectively, by $L = L_{\alpha \beta}^{cc'}$. To prove the statement (b), we need to show that for the first bad crossing $c'$ with respect to $b$ and for any crossing $c$ of $L = L^{cc'}_{\alpha\beta}$ with states $\alpha$, $\beta$ $\in \{ +,-\}$ of $c$ and $c'$ respectively, the following equation holds:
\begin{equation}\label{lem_dia}
 \widehat{W}_{b}(L^{cc'}_{-\alpha\beta}) *^{\alpha} \widehat{W}_{b}(L^{cc'}_{0\beta}) = \widehat{W}_{b}(L^{cc'}_{\alpha,-\beta}) *^{\beta} \widehat{W}_{b}(L^{cc'}_{\alpha0}),
\end{equation}
where $*^{+} =  \circ ~\text{or}~ *$ and $*^{-} = / ~\text{or}~ //$. 
Let $c$ be a crossing of $L$. Without loss of generality, we may assume that $c$ is a positive crossing of $L$ and the other cases can be proved analogously. If it is the first bad crossing of $L$, it is clear. If $c$ is a good crossing of $L_{+}^{c}$ but it is the first bad crossing of $L_{-}^{c}$, then $\widehat{W}_{b}(L_{-}^{c}) = \widehat{W}_{b}(L_{+}^{c}) /\widehat{W}_{b}(L_{0}^{c})$ (or $\widehat{W}_{b}(L_{-}^{c}) = \widehat{W}_{b}(L_{+}^{c}) //\widehat{W}_{b}(L_{0}^{c})$), that is, $\widehat{W}_{b}(L_{+}^{c}) = \widehat{W}_{b}(L_{-}^{c}) \circ \widehat{W}_{b}(L_{0}^{c})$ (or $\widehat{W}_{b}(L_{+}^{c}) = \widehat{W}_{b}(L_{-}^{c}) * \widehat{W}_{b}(L_{0}^{c})$). Suppose that $c$ is not the first crossing with respect to $b$ of $L_{+}^{c}$ and $L_{-}^{c}$. Let $c'$ be the first bad crossing of $L$. Without loss of generality, we may assume that $c'$ is a positive crossing. \\
\textbf{Case1} $c$ and $c'$ are self crossings.\\
Notice that the crossing $c$ can be changed to a mixed crossing by splicing $c'$, i.e. $c$ can be a mixed crossing of the diagram $L^{cc'}_{+0}$ even if it is a self crossing of $L^{cc'}_{++}$. 
If $c$ is a self crossing of $L^{cc'}_{+0}$, we obtain
\begin{eqnarray*}
\widehat{W}_{b}(L^{cc'}_{++}) &=& \widehat{W}_{b}(L^{cc'}_{+-}) \circ \widehat{W}_{b}(L^{cc'}_{+0})\\
 &=& (\widehat{W}_{b}(L^{cc'}_{--}) \circ \widehat{W}_{b}(L^{cc'}_{0-})) \circ (\widehat{W}_{b}(L^{cc'}_{-0}) \circ \widehat{W}_{b}(L^{cc'}_{00})) \\
 &=& (\widehat{W}_{b}(L^{cc'}_{--}) \circ \widehat{W}_{b}(L^{cc'}_{-0})) \circ (\widehat{W}_{b}(L^{cc'}_{0-}) \circ \widehat{W}_{b}(L^{cc'}_{00})) \\
 &=&\widehat{W}_{b}(L^{cc'}_{-+}) \circ \widehat{W}_{b}(L^{cc'}_{0+}),
\end{eqnarray*}
by applying the relation (\ref{selfConwayrel}). The second equation holds since $L^{cc'}_{+-}$ has $m$ bad crossings and $L^{cc'}_{+0}$ belongs to $\mathcal{L}_{k}$. The third equation holds because of the relation (C) of definition~\ref{def_genConaltype1}. The fourth equation holds, because $c'$ is the first bad crossing of $L^{cc'}_{-+}$ with respect to $b$.
If $c$ is a mixed crossing of $L^{cc'}_{+0}$, then $c'$ is also a mixed crossing when the crossing $c$ is spliced. Analogously we obtain
\begin{eqnarray*}
\widehat{W}_{b}(L^{cc'}_{++}) &=& \widehat{W}_{b}(L^{cc'}_{+-}) \circ \widehat{W}_{b}(L^{cc'}_{+0})\\
 &=& (\widehat{W}_{b}(L^{cc'}_{--}) \circ \widehat{W}_{b}(L^{cc'}_{0-})) \circ (\widehat{W}_{b}(L^{cc'}_{-0}) * \widehat{W}_{b}(L^{cc'}_{00})) \\
 &=& (\widehat{W}_{b}(L^{cc'}_{--}) \circ \widehat{W}_{b}(L^{cc'}_{-0})) \circ (\widehat{W}_{b}(L^{cc'}_{0-}) * \widehat{W}_{b}(L^{cc'}_{00})) \\
 &=&\widehat{W}_{b}(L^{cc'}_{-+}) \circ \widehat{W}_{b}(L^{cc'}_{0+}),
\end{eqnarray*}
by applying the relations (\ref{selfConwayrel}) and (\ref{mixedConwayrel}), and the relation (E) of definition~\ref{def_genConaltype1}.\\
\textbf{Case2} $c$ and $c'$ are mixed crossings.\\
Notice that the crossing $c$ can be changed to a self crossing by splicing $c'$, i.e. $c$ can be a self crossing of the diagram $L^{cc'}_{+0}$ even if it is a mixed crossing of $L^{cc'}_{++}$. 
If $c$ is a mixed crossing of $L^{cc'}_{+0}$, we obtain
\begin{eqnarray*}
\widehat{W}_{b}(L^{cc'}_{++}) &=& \widehat{W}_{b}(L^{cc'}_{+-}) * \widehat{W}_{b}(L^{cc'}_{+0})\\
 &=& (\widehat{W}_{b}(L^{cc'}_{--}) * \widehat{W}_{b}(L^{cc'}_{0-})) * (\widehat{W}_{b}(L^{cc'}_{-0}) *\widehat{W}_{b}(L^{cc'}_{00})) \\
 &=& (\widehat{W}_{b}(L^{cc'}_{--}) * \widehat{W}_{b}(L^{cc'}_{-0})) * (\widehat{W}_{b}(L^{cc'}_{0-}) *\widehat{W}_{b}(L^{cc'}_{00})) \\
 &=&\widehat{W}_{b}(L^{cc'}_{-+}) * \widehat{W}_{b}(L^{cc'}_{0+}),
\end{eqnarray*}
by applying the relation (\ref{mixedConwayrel}) and the relation (D) of definition~\ref{def_genConaltype1}.
If $c$ is a self crossing of $L^{cc'}_{+0}$, we obtain
\begin{eqnarray*}
\widehat{W}_{b}(L^{cc'}_{++}) &=& \widehat{W}_{b}(L^{cc'}_{+-}) * \widehat{W}_{b}(L^{cc'}_{+0})\\
 &=& (\widehat{W}_{b}(L^{cc'}_{--}) * \widehat{W}_{b}(L^{cc'}_{0-})) * (\widehat{W}_{b}(L^{cc'}_{-0}) \circ\widehat{W}_{b}(L^{cc'}_{00})) \\
 &=& (\widehat{W}_{b}(L^{cc'}_{--}) * \widehat{W}_{b}(L^{cc'}_{-0})) * (\widehat{W}_{b}(L^{cc'}_{0-}) \circ\widehat{W}_{b}(L^{cc'}_{00})) \\
 &=&\widehat{W}_{b}(L^{cc'}_{-+}) * \widehat{W}_{b}(L^{cc'}_{0+}),
\end{eqnarray*}
by applying the relations (\ref{selfConwayrel}) and (\ref{mixedConwayrel}), and the relation (F) of definition~\ref{def_genConaltype1}.

\textbf{Case3} $c$ is a mixed crossing and $c'$ is a self crossing.\\
We obtain
\begin{eqnarray*}
\widehat{W}_{b}(L^{cc'}_{++}) &=& \widehat{W}_{b}(L^{cc'}_{+-}) \circ \widehat{W}_{b}(L^{cc'}_{+0})\\
 &=& (\widehat{W}_{b}(L^{cc'}_{--}) * \widehat{W}_{b}(L^{cc'}_{0-})) \circ (\widehat{W}_{b}(L^{cc'}_{-0}) *\widehat{W}_{b}(L^{cc'}_{00})) \\
 &=& (\widehat{W}_{b}(L^{cc'}_{--}) \circ \widehat{W}_{b}(L^{cc'}_{-0})) * (\widehat{W}_{b}(L^{cc'}_{0-}) \circ\widehat{W}_{b}(L^{cc'}_{00})) \\
 &=&\widehat{W}_{b}(L^{cc'}_{-+}) * \widehat{W}_{b}(L^{cc'}_{0+})
\end{eqnarray*}
by applying the relations (\ref{selfConwayrel}) and (\ref{mixedConwayrel}), and the third equation holds because of the relation (G) of definition~\ref{def_genConaltype1}.\\
\textbf{Case4} $c$ is a self crossing and $c'$ is a mixed crossing.\\
This case is similar to the case 3. By the induction on the number of bad crossings with respect to $b$ the statement (b) is true for $L$ in $\mathcal{L}_{k+1}$.

\textbf{(c) The value of $\widehat{W}_{b}(L)$ does not depend on the choice of the base points.}
 Assume that $\widehat{W}_{b}(L)$ does not depend on the base points for diagrams in $\mathcal{L}_{k}$. Let $L$ be a diagram in $\mathcal{L}_{k+1}$ of $r$ components with base points $b$. To show the independence on the base points, it is sufficient to consider that one of base points is replaced by another point. Assume that $b = (b_{1}, \cdots, b_{l}, \cdots, b_{r})$ is changed to $b' = (b_{1}, \cdots, b'_{l}, \cdots, b_{r})$ satisfying that $b'_{l}$ is obtained from $b_{l}$ by passing only one crossing $c$ as described in Fig.~\ref{2replace_basept}. Without loss of generality, we may assume that the crossing $c$ is positive. 	
 \begin{figure}[h!]
\begin{center}
 \includegraphics[width = 8cm]{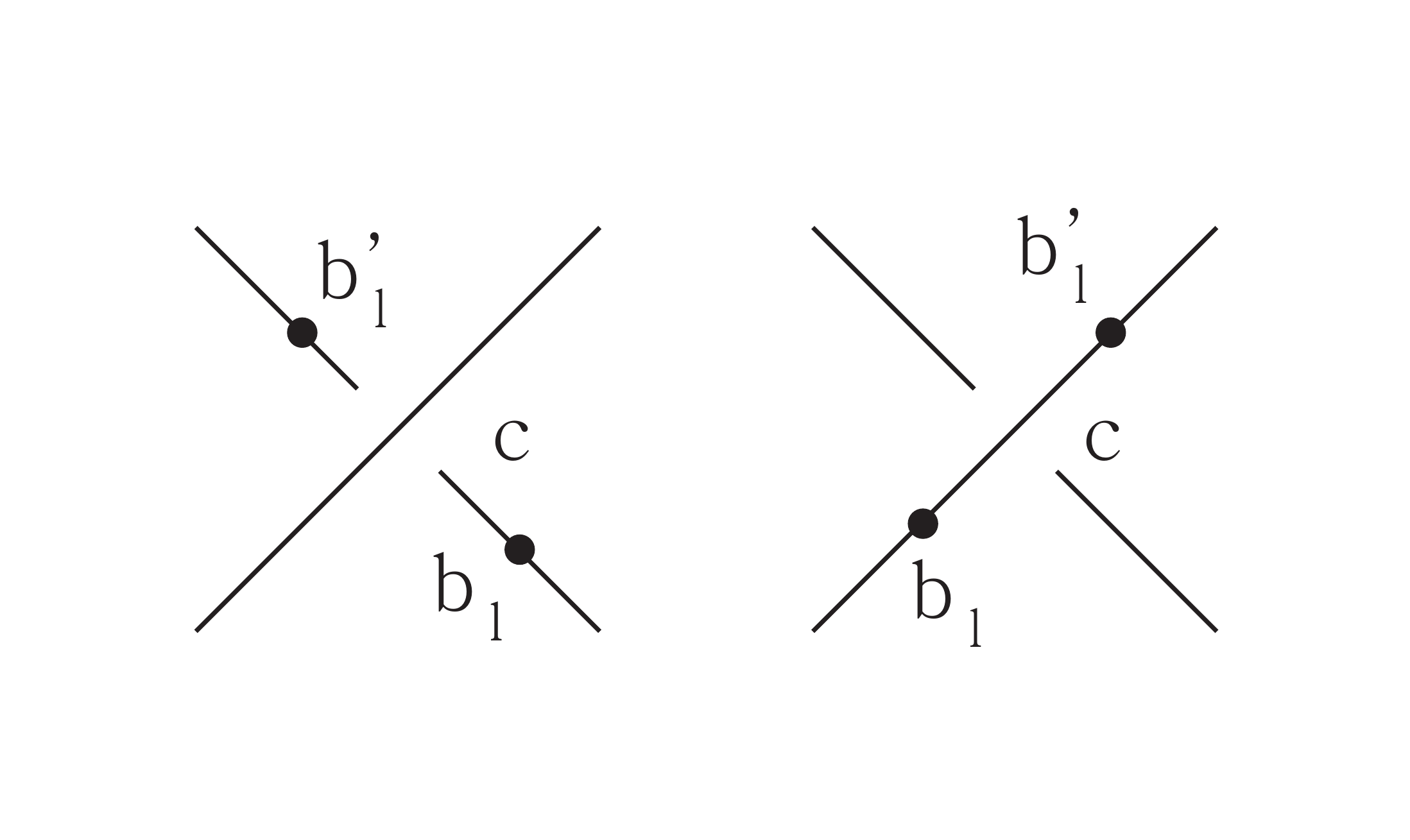}

\end{center}

 \caption{Replace of base points}\label{3replace_basept}
\end{figure}
Suppose that $L$ has no bad crossings with respect to $b$. If $c$ is a good crossing with respect to both of $b$ and $b'$, it is clear that  $\widehat{W}_{b}(L) = \widehat{W}_{b'}(L)$ because $L$ has no bad crossings with respect to both of $b$ and $b'$. Suppose that $c$ is a good crossing with respect to $b$, but $c$ is a bad crossing with respect to $b'$. Notice that $L$ has no bad crossings with respect to $b$, but $L$ has only one bad crossing with respect to $b'$. Moreover this case is possible only when $c$ is a self crossing. Then we obtain that
\begin{equation*}
\widehat{W}_{b}(L) = a_{r},
\end{equation*}
and
\begin{equation*}
\widehat{W}_{b'}(L^{c}_{+}) = \widehat{W}_{b'}(L^{c}_{-}) \circ \widehat{W}_{b'}(L^{c}_{0}).
\end{equation*}
Since the diagram $L^{c}_{0}$ belongs to $\mathcal{L}_{k}$ and the diagram $L^{c}_{+}$ has the form of a descending diagram, by the assumption of the induction, we can choose other base points $\widetilde{b'}$ on $L^{c}_{0}$ for the diagram $L^{c}_{0}$ to have no bad crossings with respect to $\widetilde{b'}$. Since $L^{c}_{-}$ has no bad crossings with respect to $b'$, we obtain that 
\begin{equation*}
\widehat{W}_{b'}(L^{c}_{+}) = \widehat{W}_{b'}(L^{c}_{-}) \circ \widehat{W}_{b'}(L^{c}_{0}) =  \widehat{W}_{b'}(L^{c}_{-}) \circ \widehat{W}_{\widetilde{b'}}(L^{c}_{0}) = a_{r} \circ a_{r+1}.
\end{equation*}
By relation (B) of definition~\ref{def_genConaltype1} the equation $\widehat{W}_{b}(L) = \widehat{W}_{b'}(L)$ holds.

Assume that $\widehat{W}_{b}(L) =\widehat{W}_{b'}(L)$ for diagrams in $\mathcal{L}_{k+1}$ with $m \ge 0$ bad crossings with respect to base points $b$. Suppose that $L$ has $m+1$ bad crossings with respect to $b$. \\
\textbf{Case1.} $c$ is a mixed crossing.\\
If $c$ is a bad mixed crossing with respect to $b$, then $c$ is also bad with respect to $b'$ and we obtain 
$$\widehat{W}_{b}(L^{c}_{+}) = \widehat{W}_{b}(L^{c}_{-}) * \widehat{W}_{b}(L^{c}_{0}) ~\text{and}~ \widehat{W}_{b'}(L^{c}_{+}) = \widehat{W}_{b'}(L^{c}_{-}) *\widehat{W}_{b'}(L^{c}_{0}).$$ Since $L^{c}_{-}$ has $m$ bad crossings with respect to $b$ and  $L^{c}_{0}$ belongs to $\mathcal{L}_{k}$, by the assumption for the second induction the equation $\widehat{W}_{b}(L^{c}_{-}) = \widehat{W}_{b'}(L^{c}_{-})$ holds, and by the assumption for the induction $\widehat{W}_{b}(L^{c}_{0})=\widehat{W}_{b'}(L^{c}_{0})$ holds. It follows that the equation $\widehat{W}_{b}(L) = \widehat{W}_{b'}(L)$ holds. If $c$ is a good crossing with respect to both of $b$ and $b'$, then there is a bad crossing $c'$ with respect to $b$ and $b'$, since $L$ has $m+1$ bad crossings with respect to $b$. Then analogously we can show that $\widehat{W}_{b}(L^{c}_{+}) =\widehat{W}_{b'}(L^{c}_{+})$.\\
\textbf{Case2.} $c$ is a self crossing.\\
If $c$ is a bad self crossing or a good self crossing with respect to both of base points $b$ and $b'$, we can show that the equation $\widehat{W}_{b}(L) = \widehat{W}_{b'}(L)$ holds as the previous case. Suppose that $c$ is a good self crossing with respect to $b$, but $c$ is a bad self crossing with respect to $b'$. Since $L$ has $m+1$ bad crossings with respect to $b$, there is another bad crossing $c'$ of $L$ with respect to $b$ and $b'$. Without loss of generality we may assume that $c'$ is a positive crossing. If $c'$ is a positive bad self crossing of $L$,
by applying the relations~(\ref{selfConwayrel}) and (\ref{mixedConwayrel}) to $c$ and $c'$ we obtain that 
\begin{eqnarray*}
\widehat{W}_{b'}(L^{cc'}_{++}) &=& \widehat{W}_{b'}(L^{cc'}_{+-}) \circ \widehat{W}_{b'}(L^{cc'}_{+0})\\
&=&( \widehat{W}_{b'}(L^{cc'}_{--})  \circ  \widehat{W}_{b'}(L^{cc'}_{0-}) ) \circ (\widehat{W}_{b'}(L^{cc'}_{-0}) \alpha \widehat{W}_{b'}(L^{cc'}_{00})),
\end{eqnarray*}
where $\alpha = \circ$ if $c$ is a self crossing of $L^{cc'}_{+0}$, or $\alpha = *$, otherwise. Since $L^{cc'}_{--}$ has $m$ bad crossings with respect to $b'$ and $L^{cc'}_{0-}$, $L^{cc'}_{-0}$ and $L^{cc'}_{00}$ have less than $k+1$ crossings,  $\widehat{W}_{b'}(L^{cc'}_{--}) =  \widehat{W}_{b}(L^{cc'}_{--})$, $\widehat{W}_{b'}(L^{cc'}_{-0}) =\widehat{W}_{b}(L^{cc'}_{-0})$, $\widehat{W}_{b'}(L^{cc'}_{0-}) = \widehat{W}_{b}(L^{cc'}_{0-})$ and $\widehat{W}_{b'}(L^{cc'}_{00}) = \widehat{W}_{b}(L^{cc'}_{00})$. Then we obtain that 
\begin{eqnarray*}
\widehat{W}_{b'}(L^{cc'}_{++}) &=&( \widehat{W}_{b'}(L^{cc'}_{--})  \circ  \widehat{W}_{b'}(L^{cc'}_{0-}) ) \circ ( \widehat{W}_{b'}(L^{cc'}_{-0}) \alpha \widehat{W}_{b'}(L^{cc'}_{00})) \\
&=&( \widehat{W}_{b}(L^{cc'}_{--})  \circ  \widehat{W}_{b}(L^{cc'}_{0-}) ) \circ ( \widehat{W}_{b}(L^{cc'}_{-0}) \alpha \widehat{W}_{b}(L^{cc'}_{00})) \\
&=& \widehat{W}_{b}(L^{cc'}_{+-}) \circ \widehat{W}_{b}(L^{cc'}_{+0}) \\
&=& \widehat{W}_{b}(L^{cc'}_{++}).
\end{eqnarray*}
If $c'$ is a positive bad mixed crossing, analogously we obtain that 
\begin{eqnarray*}
\widehat{W}_{b'}(L^{cc'}_{++}) &=& \widehat{W}_{b'}(L^{cc'}_{+-}) * \widehat{W}_{b'}(L^{cc'}_{+0})\\
&=&(\widehat{W}_{b'}(L^{cc'}_{--})  \circ  \widehat{W}_{b'}(L^{cc'}_{0-}) ) * (\widehat{W}_{b'}(L^{cc'}_{-0}) \circ \widehat{W}_{b'}(L^{cc'}_{00})) \\
&=&( \widehat{W}_{b}(L^{cc'}_{--})  \circ  \widehat{W}_{b}(L^{cc'}_{0-}) ) * ( \widehat{W}_{b}(L^{cc'}_{-0}) \circ \widehat{W}_{b}(L^{cc'}_{00})) \\
&=& \widehat{W}_{b}(L^{cc'}_{+-}) * \widehat{W}_{b}(L^{cc'}_{+0}) \\
&=& \widehat{W}_{b}(L^{cc'}_{++}).
\end{eqnarray*}
By the second induction on the number of bad crossings $\widehat{W}_{b}$ does not depend on the choice of base points for diagrams $L$ in $\mathcal{L}_{k+1}$. Since $\widehat{W}_{b}$ is independent to the choice of base points, now we consider a function $\widehat{W}_{0}$, which associates an element of $\widehat{\mathcal{A}}$ to any diagram $L$ in $\mathcal{L}_{k+1}$ with a fixed order of components.

\textbf{(d) The value of $\widehat{W}_{0}(L)$ is invariant under Reidemeister moves, which do not make the number of crossings more than $k+1$.}
Let $L$ be an ordered oriented link diagram of $r$ components. Let $R$ be one of Reidemeister moves on diagrams in $\mathcal{L}_{k+1}$. The diagram $R(L)$, which is obtained from $L$ by applying $R$, has the natural order of components. Let us denote the number of crossings of a diagram $L$ by $cr(L)$. Assume that $cr(R(L)) \leq cr(L) \leq k+1$. Since $\widehat{W}_{0}(L)$ does not depend on the base points, we may assume that base points $b$ are chosen outside of the part of the diagram involved in the Reidemeister move $R$ and are compatible with the given order of components. We will show that $\widehat{W}_{0}(L) = \widehat{W}_{0}(R(L))$. We use the second induction on the number of bad crossings $b(L)$ for base points $b$. Assume that $b(L) =0$. It is clear that $b(R(L)) =0$ and the number of components remains. Therefore $\widehat{W}_{0}(L)= a_{r} = \widehat{W}_{0}(R(L))$. Now we assume that $\widehat{W}_{0}(L) = \widehat{W}_{0}(R(L))$ for diagrams $L$ with $b(L) \leq m$. Suppose that $b(L) =m+1$. 
Assume that there is a bad pure crossing which is not involved in the considered Reidemeister moves. We may assume that the bad crossing $c$ is positive. By the assumption for the second induction, we obtain that 
\begin{equation*}
\widehat{W}_{0}(L_{-}^{c}) = \widehat{W}_{0}(R(L_{-}^{c})).
\end{equation*}
Since $cr(L^{c}_{+})>cr(L^{c}_{0})$, by the assumption for induction, we obtain that
\begin{equation*}
\widehat{W}_{0}(L_{0}^{c}) = \widehat{W}_{0}(R(L_{0}^{c})).
\end{equation*}
By the relation~(\ref{selfConwayrel}), we obtain
\begin{equation*}
\widehat{W}_{0}(L_{+}^{c}) = \widehat{W}_{0}(L_{-}^{c}) \circ \widehat{W}_{0}(L_{0}^{c})
\end{equation*}
and 
\begin{equation*}
\widehat{W}_{0}(R(L)_{+}^{c}) = \widehat{W}_{0}(R(L)_{-}^{c}) \circ \widehat{W}_{0}(R(L)_{0}^{c}).\end{equation*}
It is clear that $R(L^{c}_{-}) = R(L)^{c}_{-}$ and  $R(L^{c}_{0}) = R(L)^{c}_{0}$ and hence we obtain 

\begin{equation*}
\widehat{W}_{0}(L_{+}^{c}) = \widehat{W}_{0}(R(L_{+}^{c})).
\end{equation*}
If there is a bad mixed crossing $c$, which is not contained in the Reidemeister move $R$, analogously we can obtain 
\begin{equation*}
\widehat{W}_{0}(L_{+}^{c}) = \widehat{W}_{0}(R(L_{+}^{c})).
\end{equation*}
Now we assume that every bad crossing is contained in $R$. Suppose that a bad pure crossing $c$ is contained in the first Reidemeister move $R$. Without loss of generality, we may assume that $c$ is a positive crossing. If there is another bad crossing $c'$, then analogously we obtain that $\widehat{W}_{b}(R(L)) = \widehat{W}_{b}(L)$, since $c'$ must not be contained in the first Reidemeister move. It is sufficient to show that if $L$ has the only one bad crossing $c$ with respect to the given base points $b$, then $\widehat{W}_{b}(R(L)) = \widehat{W}_{b}(L)$. Suppose that $L$ has the only one bad crossing $c$ with respect to $b$. Since $c$ is the only one bad crossing with respect to $b$ of $L_{+}^{c}$, $L_{-}^{c}$ has no bad crossings and hence $\widehat{W}_{b}(L_{-}^{c}) = a_{r} = \widehat{W}_{b}(R(L_{-}^{c}))$, where $r$ is the number of components of $L$. Since $L_{0}^{c} = R(L_{-}^{c}) \sqcup T_{1}$, where $T_{1}$ is the trivial circle, $L_{0}^{c}$ has no bad crossings with respect to $b$ and $\widehat{W}_{b}(L_{0}^{c}) = a_{r+1}$, see Fig.~\ref{4inv_RM1}. 
\begin{figure}[h!]
\begin{center}
 \includegraphics[width = 8cm]{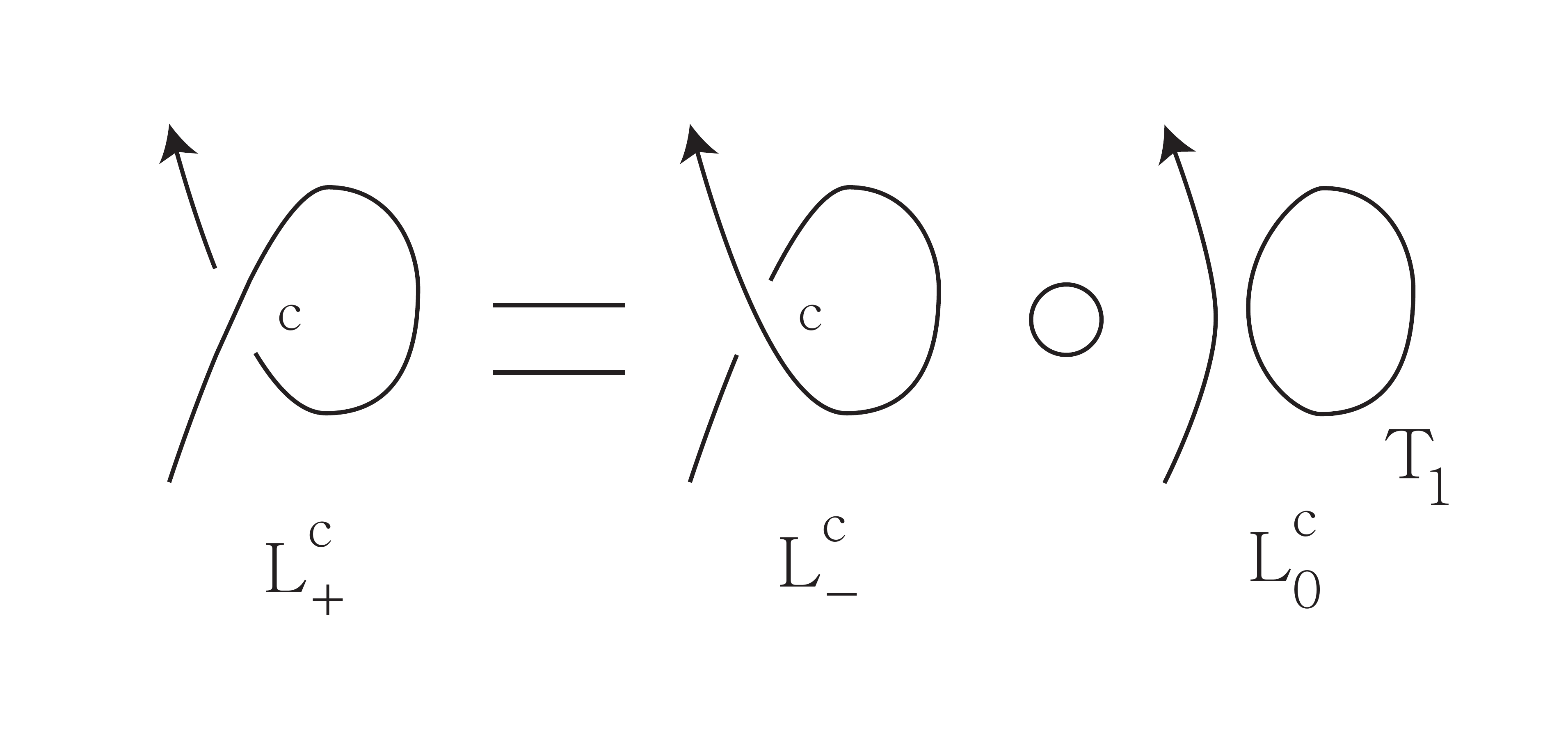}

\end{center}

 \caption{The skein relation on bad pure crossing in RM1}\label{4inv_RM1}
\end{figure}
By applying the relation~(\ref{selfConwayrel}) we obtain that $\widehat{W}_{b}(L_{+}^{c}) = \widehat{W}_{b}(L_{-}^{c}) \circ \widehat{W}_{b}(L_{0}^{c}) = a_{r} \circ a_{r+1}$. Since $R(L)$ has no bad crossings, $\widehat{W}_{b}(R(L)) = a_{r} = a_{r} \circ a_{r+1} = \widehat{W}_{b}(L_{+}^{c})=\widehat{W}_{b}(L)$.

If a bad pure crossing is contained in the second Reidemeister move, then another crossing is also bad pure crossing with the opposite crossing sign. If $c$ and $c'$ are bad pure crossings contained in the second Reidemeister move, we obtain that
\begin{eqnarray*}
\widehat{W}_{0}(L^{cc'}_{+-}) &=& \widehat{W}_{0}(L^{cc'}_{--}) \circ \widehat{W}_{0}(L^{cc'}_{0-})\\
&=&  (\widehat{W}_{0}(L^{cc'}_{-+}) / \widehat{W}_{0}(L^{cc'}_{-0})) \circ \widehat{W}_{0}(L^{cc'}_{0-}) \\
&=&  (\widehat{W}_{0}(L^{cc'}_{-+}) / \widehat{W}_{0}(L^{cc'}_{0-})) \circ \widehat{W}_{0}(L^{cc'}_{0-}) \\
&=& \widehat{W}_{0}(L^{cc'}_{-+}).
\end{eqnarray*}
The third equation holds because the diagrams $L^{cc'}_{-0}$ and $L^{cc'}_{0-}$ have the same diagram. Since $b(L^{cc'}_{-+}) = b(L^{cc'}_{+-})-2$ and $ \widehat{W}_{0}(R(L^{cc'}_{-+})) =  \widehat{W}_{0}(R(L^{cc'}_{+-}))$, we get $$\widehat{W}_{0}(L^{cc'}_{-+}) = \widehat{W}_{0}(R(L^{cc'}_{-+})) = \widehat{W}_{0}(R(L^{cc'}_{+-})) = \widehat{W}_{0}(L^{cc'}_{+-}),$$ see Fig.~\ref{5pure_inv_RM2}.
\begin{figure}[h!]
\begin{center}
 \includegraphics[width = 8cm]{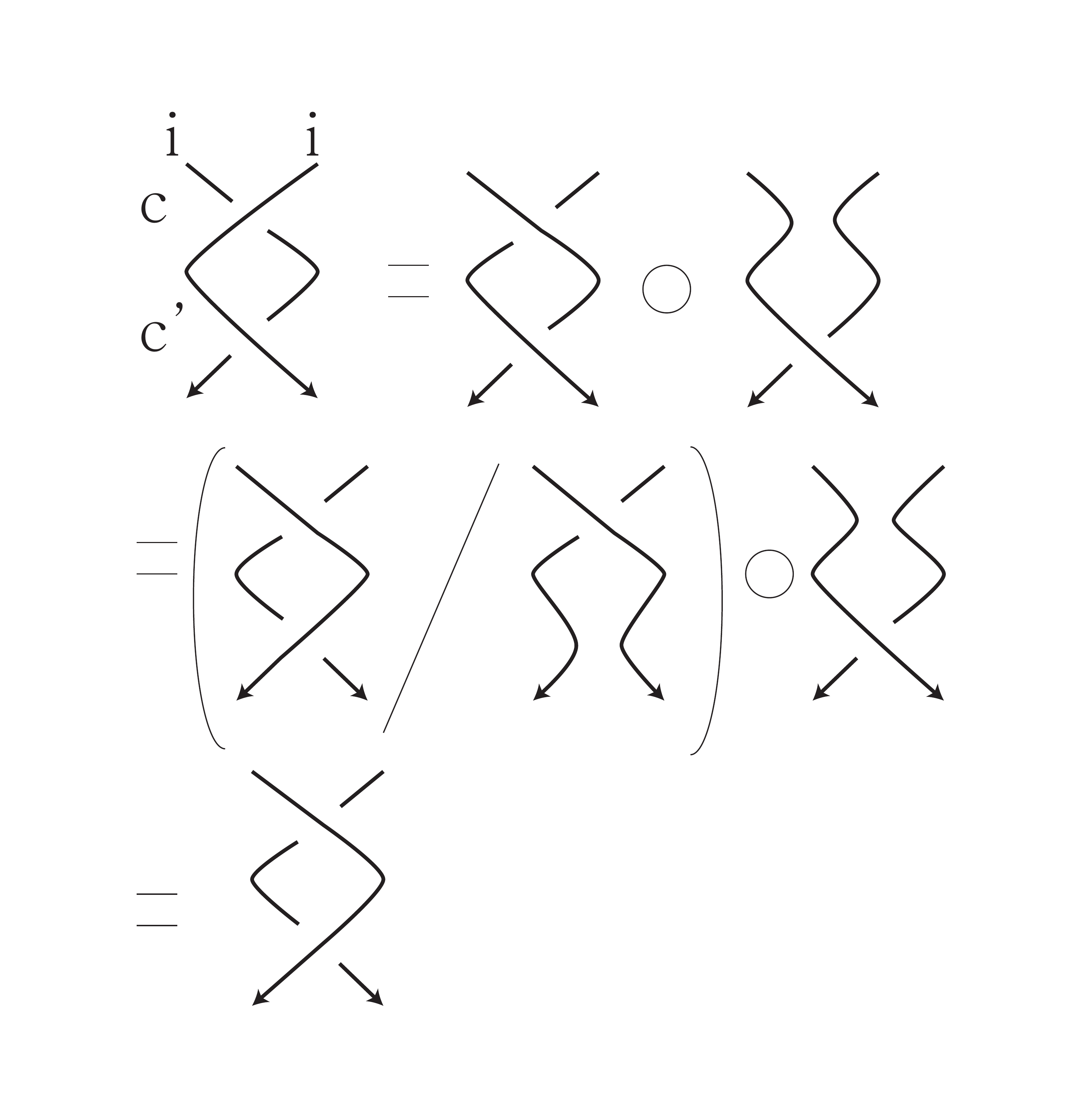}

\end{center}

 \caption{The skein relation on bad pure crossings in RM2}\label{5pure_inv_RM2}
\end{figure}
If $c$ and $c'$ are bad mixed crossings, then analogously we obtain that $$\widehat{W}_{0}(L^{cc'}_{+-}) = \widehat{W}_{0}(R(L^{cc'}_{+-})).$$

Suppose that $c$ is a bad crossing in the third Reidemeister move $R$. Note that if the crossing between the top arc and the bottom arc is bad, then there is another bad crossing contained in $R$. That is, it suffices to show that for a crossing $c$ between the top and the middle arcs(or the middle and the bottom arcs) the equation $\widehat{W}_{0}(L^{c}_{+}) = \widehat{W}_{0}(R(L^{c}_{+}))$ holds. Assume that $c$ is the crossing between the middle and the bottom arcs in $R$.
If $c$ is a positive pure crossing, then by the skein relation, we obtain
\begin{equation*}
\widehat{W}_{0}(L^{c}_{+}) = \widehat{W}_{0}(L^{c}_{-}) \circ \widehat{W}_{0}(L^{c}_{0})
\end{equation*}
and
\begin{equation*}
\widehat{W}_{0}(R(L)^{c}_{+}) = \widehat{W}_{0}(R(L)^{c}_{-}) \circ \widehat{W}_{0}(R(L)^{c}_{0}).
\end{equation*}
Since $L^{c}_{-}$ and $R(L)^{c}_{-}$ have $m$ bad crossings, by the assumption for the second induction, $\widehat{W}_{0}(L^{c}_{-}) = \widehat{W}_{0}(R(L)^{c}_{-})$. Since $\widehat{W}_{0}(L^{c}_{0})$ and $\widehat{W}_{0}(R(L)^{c}_{0})$ have the same diagram as described in Fig.~\ref{6pure_inv_RM3-1}, we obtain that $\widehat{W}_{0}(L^{c}_{+}) = \widehat{W}_{0}(R(L^{c}_{+}))$. 
\begin{figure}[h!]
\begin{center}
 \includegraphics[width = 10cm]{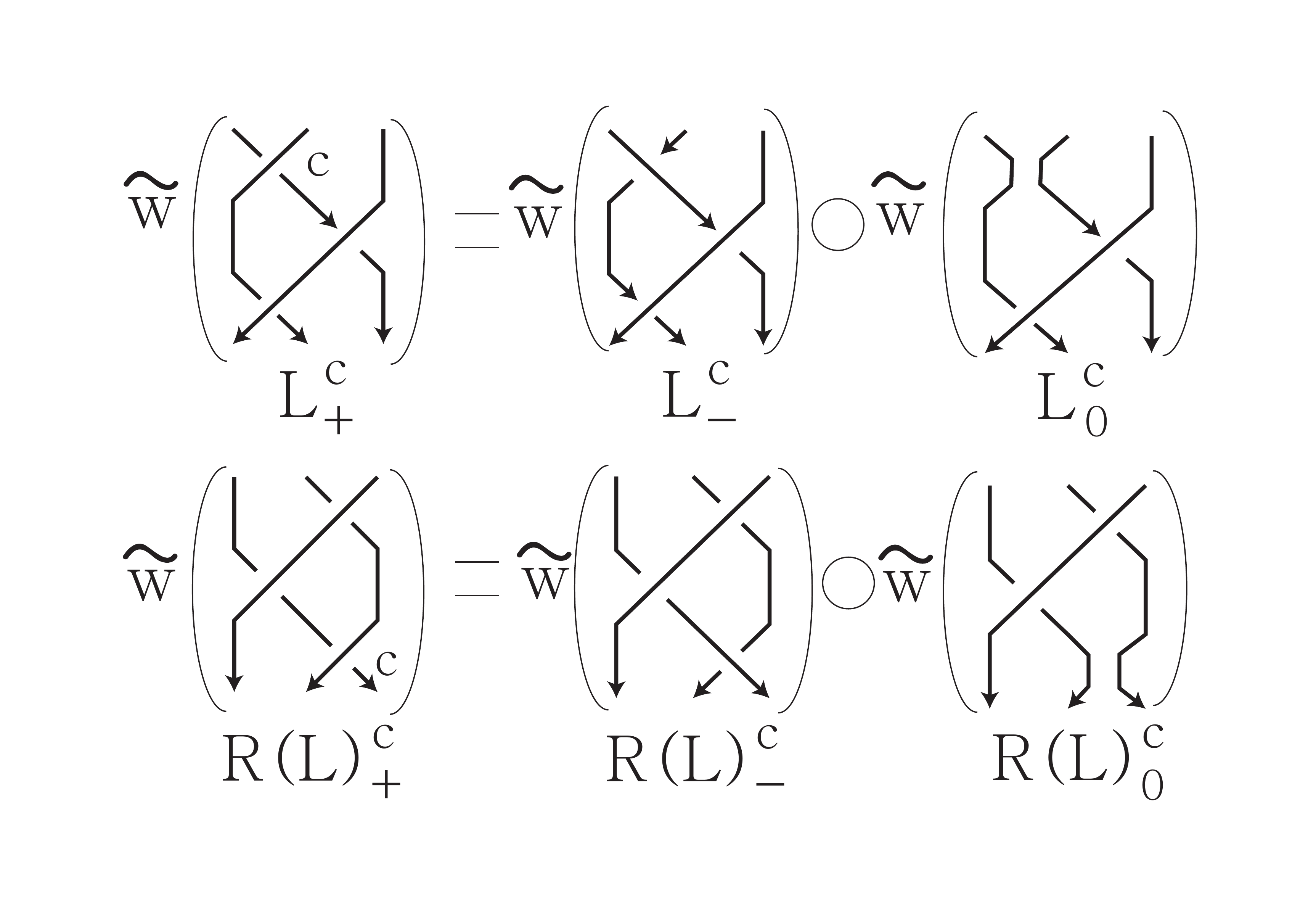}

\end{center}

 \caption{The skein relation on a positive bad pure crossing in RM3}\label{6pure_inv_RM3-1}
\end{figure}
If $c$ is a negative pure crossing, we obtain 
\begin{equation*}
\widehat{W}_{0}(L^{c}_{-}) = \widehat{W}_{0}(L^{c}_{+}) / \widehat{W}_{0}(L^{c}_{0})
\end{equation*}
and
\begin{equation*}
\widehat{W}_{0}(R(L)^{c}_{-}) = \widehat{W}_{0}(R(L)^{c}_{+}) / \widehat{W}_{0}(R(L)^{c}_{0}).
\end{equation*}
Since $L^{c}_{+}$ and $R(L)^{c}_{+}$ has $m$ bad crossings, by the assumption for the second induction, the equation $\widehat{W}_{0}(L^{c}_{+}) = \widehat{W}_{0}(R(L)^{c}_{+})$ holds. Since we already proved that $\widehat{W}_{0}$ is invariant under the second Reidemeister move, we obtain that $\widehat{W}_{0}(L^{c}_{0}) = \widehat{W}_{0}(R(L)^{c}_{0})$, see Fig.~\ref{7pure_inv_RM3-2}. Therefore $\widehat{W}_{0}(L^{c}_{-}) = \widehat{W}_{0}(R(L)^{c}_{-})$. 
\begin{figure}[h!]
\begin{center}
 \includegraphics[width = 10cm]{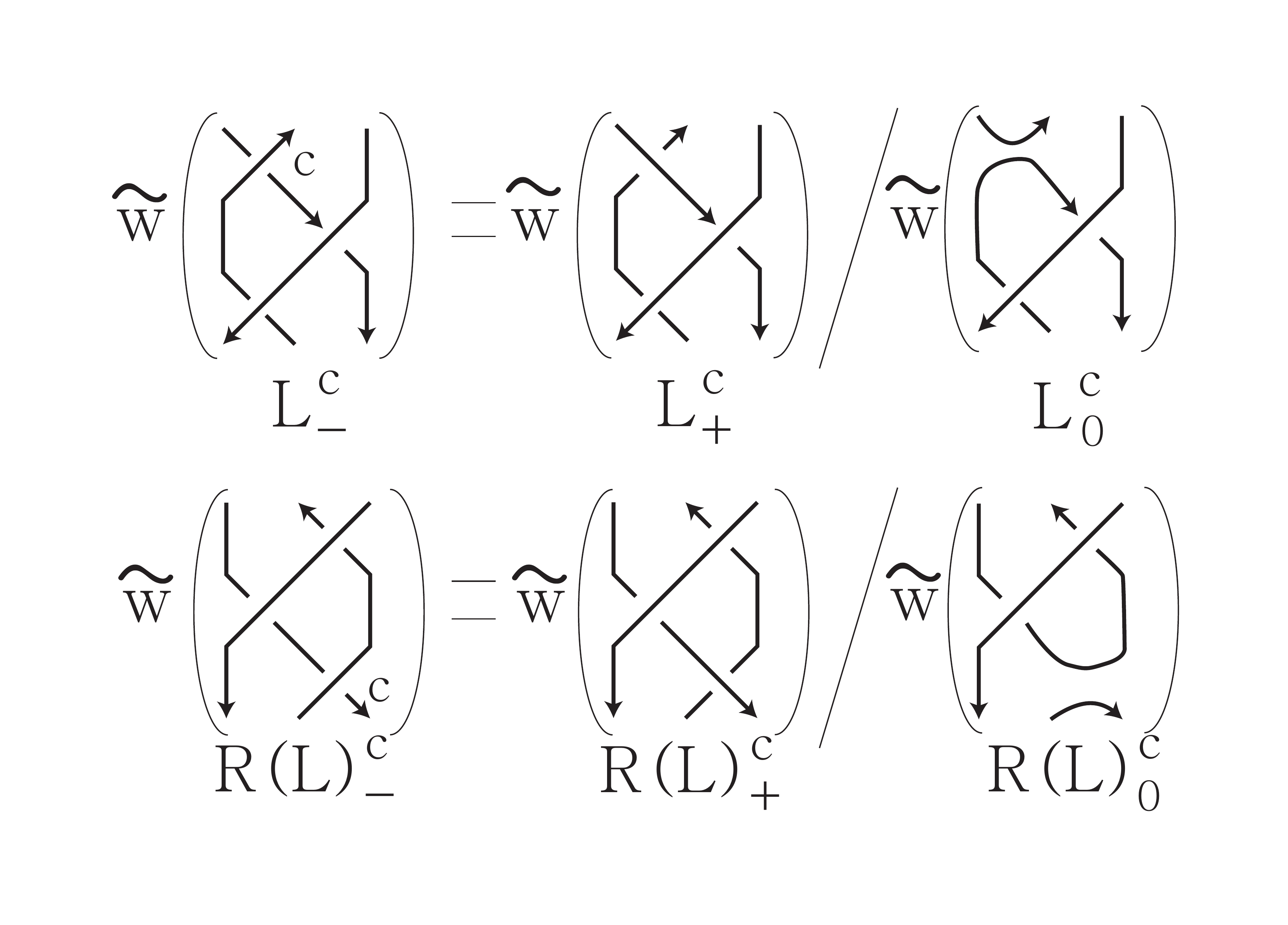}

\end{center}

 \caption{The skein relation on a negative bad pure crossing in RM3}\label{7pure_inv_RM3-2}
\end{figure}
If $c$ is a mixed crossing, then analogously we obtain that $\widehat{W}_{0}(L) = \widehat{W}_{0}(R(L))$ and the statement is true for diagrams in $\mathcal{L}_{k+1}$.

\textbf{(e) The value of $\widehat{W}_{0}(L)$ does not depend on the order of components.}

For an ordered oriented link diagram $L = L_{1} \cup \cdots \cup L_{r}$ of $r$ components, let $b = \{b_{1},\cdots, b_{i},b_{i+1}, \cdots, b_{r}\}$ be ordered base points. If $L$ is a knot diagram, then it is clear that the value of $\widehat{W}_{0}(L)$ does not depend on the order of components. Now we assume that $r>1$. Suppose that $b' = \{b_{1},\cdots, b_{i+1},b_{i}, \cdots, b_{r}\}$ is ordered base points on $L$ obtained from $b$ by changing the order of $L_{i}$ and $L_{i+1}$. Let us consider a diagram $L$ which has a trivial component separated from other components. By relocating the trivial component we obtain a new diagram $L'$. If the relocated component is also separated from other components for the diagram $L'$, then we denote the diagram $L'$ by $RLT(L)$, for example, see Fig.~\ref{8RLToper}. We call this process for link diagrams which have trivial components {\it the $RLT$ operation.}
\begin{figure}[h!]
\begin{center}
 \includegraphics[width = 10cm]{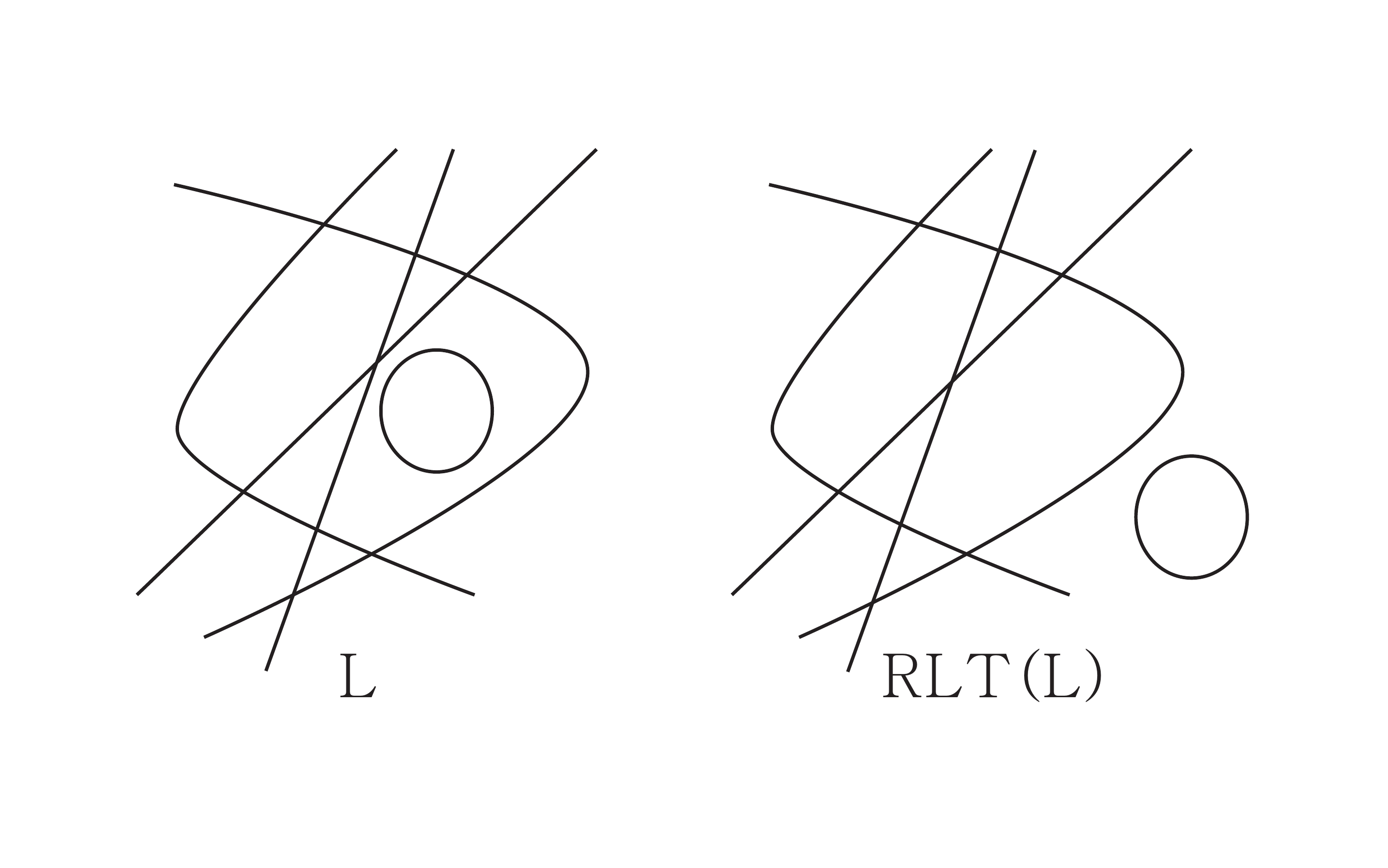}

\end{center}

 \caption{The RLT operation}\label{8RLToper}
\end{figure}

\begin{lem}\label{go_trivial}
Let $L$ be an oriented link diagram. Then $L$ can be changed to a trivial link diagram without crossings by a sequence of crossing changes, Reidemeister moves, which do not increase the number of crossings, and the $RLT$ operations. 
\end{lem}

We will prove this lemma later. Now we come back to the induction on the number of crossings of diagrams to show that the statement (e) is true for diagrams in $\mathcal{L}_{k+1}$ with base points $b$. Notice that the equation $\widehat{W}_{b}(L) = \widehat{W}_{b'}(L)$ holds if and only if the equation $\widehat{W}_{b}(L') = \widehat{W}_{b'}(L')$ holds, where $L'$ is a link diagram obtained from $L$ by crossing changes. Since $\widehat{W}_{b}$ is invariant under Reidemeister moves in $\mathcal{L}_{k+1}$, which do not increase the number of crossings, $\widehat{W}_{b}(L) = \widehat{W}_{b'}(L)$ holds if and only if $\widehat{W}_{b}(R(L)) = \widehat{W}_{b'}(R(L))$ holds. We remark that if one of components of $L$ is a trivial knot diagram $L_{k}$ and it is separated from other components, then the value of $\widehat{W}_{b}$ does not depend on the position of $L_{k}$, that is, $\widehat{W}_{b}(L)=\widehat{W}_{b}(RLT(L))$. By the previous remark it is easy to see that
\begin{center}
$\widehat{W}_{b}(L) = \widehat{W}_{b'}(L)$ if and only if $\widehat{W}_{b}(RLT(L)) = \widehat{W}_{b'}(RLT(L))$,
\end{center}
for a diagram $L$ which has a trivial component separated from other components. Since the equation $\widehat{W}_{b}(T_{n}) = \widehat{W}_{b'}(T_{n})$ holds for a trivial link diagram $T_{n}$ of $n$ components without crossings, by Lemma~\ref{go_trivial} the proof is completed.
\end{proof}
\begin{proof}[Proof of Lemma~\ref{go_trivial}]
Let $L$ be a link diagram. We call a 2-gon(a curl) of $L$ {\it an innermost} if there are no 2-gons and curls inside the 2-gon(a curl). Suppose that there is an innermost 2-gon. If there are no trivial circle, then $L$ can be deformed by crossing changes and the third Reidemeister moves to a diagram $L'$, in which there are no crossings inside the 2-gon. By a decreasing second Reidemeister move, we can remove the innermost 2-gon. If there is a trivial circle inside of the innermost 2-gon, then we relocate the trivial component outside of 2-gon to be separated from other components and then analogously we can remove the innermost 2-gon. Similarly, we can remove curls and by repeating this process and we obtain a trivial link diagram without crossings.
\end{proof}

\begin{exa}\label{exa_gen_Conway}
Let $\widehat{\mathcal{A}} = \mathbb{Z} [p^{\pm 1 }, q^{\pm 1}, r]$. Define the binary operations $\circ,*,/$ and $//$ by
\begin{eqnarray*}
a \circ b = pa + qb,& a / b = p^{-1} a - p^{-1}qb,\\
a * b = pa + rb,& a // b = p^{-1} a - p^{-1}rb.
\end{eqnarray*}
Denote $a_{n} = (\frac{1-p}{q})^{n-1}$ for each $n$. Then $(\widehat{\mathcal{A}}, \circ,/,*, //,\{a_{n}\}_{n=1}^{\infty})$ is a generalized Conway algebra of type $1$.
\end{exa}

\begin{exa}\label{exa_gen_Homflypt}
Let $\widehat{\mathcal{A}} = Z[v^{\pm 1}, z, w^{\pm 1}]$ be an algebra. Define binary operations $\circ,/,*$ and $//$ by
\begin{eqnarray*}
a\circ b = v^{2}a+vwb,& a/b = v^{-2}a-v^{-1}wb,\\ 
a* b = v^{2}a+vzb,& a//b = v^{-2}a-v^{-1}zb.
\end{eqnarray*}
Put $a_{n} = (\frac{v^{-1}-v}{w})^{n-1}.$ Then $(\widehat{\mathcal{A}},\circ,/,*,//,\{a_{n}\}_{n=1}^{\infty})$ is a generalized Conway algebra of type 1. In fact, this is obtained from the generalized Conway algebra of type 1 in example~\ref{exa_gen_Conway} by substituting $p=v^{2}$, $q = vw$ and $r = vz$. Moreover the Conway type invariant valued in the Conway algebra $(\widehat{\mathcal{A}},\circ,/,\{a_{n}\}_{n=1}^{\infty})$ is the Homflypt polynomial.
\end{exa}

On the other hands, usually polynomial invariants are defined with skein relations in the form of linear functions $L_{+} = f(L_{-},L_{0}) = pL_{-} + qL_{0}$. But there exists a generalized Conway algebra of type 1 with binary operation in the form of non-linear functions, see example~\ref{exa_non_linear_gen_Conway}. 
\begin{exa}\label{exa_non_linear_gen_Conway}
Let us fix a natural number $k$. Let $\widehat{\mathcal{A}}$ be the commutative ring with identity containing $\mathbb{Z}[p^{\pm1},q^{\pm1},r^{\pm1}]$ such that $\sqrt[k]{f} \in \widehat{\mathcal{A}}$ for each $f \in \widehat{\mathcal{A}}$, where $\sqrt[k]{f}$ is the formal $k-$th root, that is, $(\sqrt[k]{f})^{k} = \sqrt[k]{f^{k}} =f$ and $(\sqrt[k]{f})(\sqrt[k]{g}) = (\sqrt[k]{fg})$ for $f,g \in \widehat{\mathcal{A}}$. Define binary operations $\circ,/,*$ and $//$ by
\begin{eqnarray*}
a\circ b = \sqrt[k]{pa^{k}+qb^{k}}, & a/b =  \sqrt[k]{p^{-1}a^{k}-p^{-1}qb^{k}} \\
a * b = \sqrt[k]{pa^{k}+rb^{k}}, & a//b =  \sqrt[k]{p^{-1}a^{k}-p^{-1}rb^{k}},
\end{eqnarray*}
for $a,b \in \widehat{\mathcal{A}}$. Let $\{a_{n}\}$ be the sequence defined by the following recurrence relation
$$ a_{1} = 1, a_{n+1}^{k} = \frac{(1-p)}{q}a_{n}^{k}.$$
We can show that $(\widehat{\mathcal{A}}, \circ,/,*, //,\{a_{n}\}_{n=1}^{\infty})$ is a generalized Conway algebra of type $1$, see Appendix.
\end{exa}

That is, the Conway type invariant $\widehat{W}$ of type 1 valued in the generalized Conway algebra of type 1 in example~\ref{exa_non_linear_gen_Conway} satisfies the non linear skein relations
$$\widehat{W}(L^{c}_{+})= \sqrt[k]{p(\widehat{W}(L^{c}_{-}))^{k}+q(\widehat{W}(L^{c}_{0}))^{k}},$$
for a pure crossing of a diagram $L$, and
$$\widehat{W}(L^{c}_{+})= \sqrt[k]{p(\widehat{W}(L^{c}_{-}))^{k}+r(\widehat{W}(L^{c}_{0}))^{k}},$$
for a mixed crossing of a diagram $L$.
Since the 4-variable invariant of Kauffman and Lambropoulou satisfies the skein relation in the linear form, a different invariant of oriented link with the invariant of Kauffman and Lambropoulou must be obtained, if $k \not=1$. It is expected that the invariant satisfying non linear skein relation would give us more information of oriented links.

\section*{Acknowledgement}
I would like to express gratitude to my supervisor Professor Vassily Olegovich Manturov. 
Also I would like to thank Igor Mikhailovich Nikonov for helpful discussion and kind advices.

\section*{Appendix}
Let us fix a natural number $k$. Let $\widehat{\mathcal{A}}$ be the commutative ring with identity containing $\mathbb{Z}[p^{\pm1},q^{\pm1},r^{\pm1}]$ such that $\sqrt[k]{f} \in \widehat{\mathcal{A}}$ for each $f \in \widehat{\mathcal{A}}$, where $\sqrt[k]{f}$ is the formal $k-$th root, that is, $(\sqrt[k]{f})^{k} = \sqrt[k]{f^{k}} =f$ and $(\sqrt[k]{f})(\sqrt[k]{g}) = (\sqrt[k]{fg})$ for $f,g \in \widehat{\mathcal{A}}$. Define binary operations $\circ,/,*$ and $//$ by
\begin{eqnarray*}
a\circ b = \sqrt[k]{pa^{k}+qb^{k}}, & a/b =  \sqrt[k]{p^{-1}a^{k}-p^{-1}qb^{k}} \\
a * b = \sqrt[k]{pa^{k}+rb^{k}}, & a//b =  \sqrt[k]{p^{-1}a^{k}-p^{-1}rb^{k}},
\end{eqnarray*}
for $a,b \in \widehat{\mathcal{A}}$. Let $\{a_{n}\}$ be the sequence defined by the following recurrence relation
$$ a_{1} = 1, a_{n+1}^{k} = \frac{(1-p)}{q}a_{n}^{k}.$$
We will show that $(\widehat{\mathcal{A}}, \circ,/,*, //,\{a_{n}\}_{n=1}^{\infty})$ satisfies the relations (A) -- (G) in definition \ref{def_genConaltype1}. For $a,b \in \widehat{\mathcal{A}}$ we obtain that
\begin{eqnarray*}
(a \circ b) / b &=& \sqrt[k]{ p^{-1}( \sqrt[k]{ pa^{k}+qb^{k} })^{k} - p^{-1}qb^{k} } \\
 &= & \sqrt[k]{p^{-1}(pa^{k}+qb^{k}) - p^{-1}qb^{k} } \\
 &=& a.
\end{eqnarray*}
Analogously we can show that $(a/b)\circ b = a = (a* b) // b  = (a // b) * b$ and hence the binary operations satisfy the condition (A). From the condition (B), we obtain that 
\begin{alignat*}{3}
a_{n} &= a_{n} \circ a_{n+1} && \Leftrightarrow \\
a_{n} &= \sqrt[k]{pa^{k}_{n}+qa^{k}_{n+1}} && \Leftrightarrow \\
a_{n}^{k} &= pa^{k}_{n}+qa^{k}_{n+1} && \Leftrightarrow \\
(1-p)a_{n}^{k}&= qa_{n+1}^{k}.&
\end{alignat*}
Since the sequence $\{a_{n}\}_{n=1}^{\infty}$ satisfies the relation $a_{n+1}^{k} = \frac{(1-p)}{q}a_{n}^{k}$, the binary operations and the sequence $\{a_{n}\}_{n=1}^{\infty}$ satisfy the condition (B). For $a,b,c,d \in \widehat{\mathcal{A}}$, we obtain that  
\begin{eqnarray*}
(a \circ b) \circ (c \circ d) &=& \sqrt[k]{ p( \sqrt[k]{ pa^{k}+qb^{k} })^{k} + q( \sqrt[k]{ pc^{k}+qd^{k} })^{k}  } \\
 &= &\sqrt[k]{ p^{2}a^{k}+pqb^{k} + qpc^{k}+q^{2}d^{k}  },
\end{eqnarray*}
and
\begin{eqnarray*}
(a \circ c) \circ (b \circ d) &=& \sqrt[k]{ p( \sqrt[k]{ pa^{k}+qc^{k} })^{k} + q( \sqrt[k]{ pb^{k}+qd^{k} })^{k}  } \\
 &= &\sqrt[k]{ p^{2}a^{k}+pqc^{k} + qpb^{k}+q^{2}d^{k}  }.
\end{eqnarray*}
Therefore we can show that the condition (C) $(a \circ b) \circ (c \circ d) = (a \circ c) \circ (b \circ d)$ holds. Similarly, we can show that the condition (D) $(a * b) * (c * d) = (a * c) * (b * d)$ holds. For the condition (E), let $a,b,c,d \in \widehat{\mathcal{A}}$. We obtain that 
\begin{eqnarray*}
(a \circ b) \circ (c * d) &=& \sqrt[k]{ p( \sqrt[k]{ pa^{k}+qb^{k} })^{k} + q( \sqrt[k]{ pc^{k}+rd^{k} })^{k}  } \\
 &= &\sqrt[k]{ p^{2}a^{k}+pqb^{k} + qpc^{k}+qrd^{k}  },
\end{eqnarray*}
and 
\begin{eqnarray*}
(a \circ c) \circ (b * d) &=& \sqrt[k]{ p( \sqrt[k]{ pa^{k}+qc^{k} })^{k} + q( \sqrt[k]{ pb^{k}+rd^{k} })^{k}  } \\
 &= &\sqrt[k]{ p^{2}a^{k}+pqc^{k} + qpb^{k}+qrd^{k} }.
\end{eqnarray*}
We obtain that $(a \circ b) \circ (c * d) = (a \circ b) \circ (c * d)$ and hence the condition (E) holds. Analogously the condition (F) $(a * b) * (c \circ d) = (a * b) * (c \circ d)$ holds. Finally we show that the binary operations satisfies the condition (G). Let $a,b,c,d \in \widehat{\mathcal{A}}$. Analogously we obtain that 
\begin{eqnarray*}
(a \circ b) * (c \circ d) &=& \sqrt[k]{ p( \sqrt[k]{ pa^{k}+qb^{k} })^{k} + r( \sqrt[k]{ pc^{k}+qd^{k} })^{k}  } \\
 &= &\sqrt[k]{ p^{2}a^{k}+pqb^{k} + rpc^{k}+rqd^{k}  },
\end{eqnarray*}
and
\begin{eqnarray*}
(a * c) \circ (b * d) &=& \sqrt[k]{ p( \sqrt[k]{ pa^{k}+rc^{k} })^{k} + q( \sqrt[k]{ pb^{k}+rd^{k} })^{k}  } \\
 &= &\sqrt[k]{ p^{2}a^{k}+prc^{k} +qpb^{k}+qrd^{k}  }.
\end{eqnarray*}
Therefore we can see that the condition (G) $(a \circ b) * (c \circ d) = (a * c) \circ (b * d)$ holds.

\end{document}